\documentclass[11pt,a4paper,reqno]{amsart}
\usepackage{xypic,a4wide,amsmath,amssymb,amsthm,mathtools}
\mathtoolsset{showonlyrefs}
\usepackage[mathcal]{euscript}
\usepackage{mathrsfs}
\let\mathcal=\mathscr

\newtheorem{thm}{Theorem}[section]

\newtheorem{lemma}[thm]{Lemma}
\newtheorem{prop}[thm]{Proposition}
\newtheorem{defin}[thm]{Definition}
\newtheorem{assumption}[thm]{Assumption}
\theoremstyle{remark}
\newtheorem{rem}[thm]{Remark}
\newtheorem{ex}[thm]{Example}
\newtheorem{war}[thm]{Warning}

\newenvironment{remark}{\begin{rem}\rm}{\qee\end{rem}}
\newenvironment{example}{\begin{ex}\rm}{\qee\end{ex}}
\newcommand{\End}{\mbox{\it End}\,}
\newcommand{\cA}{{\mathcal A}}
\newcommand{\cB}{{\mathcal B}}

\newcommand{\calD}{{\mathcal D}}
\newcommand{\cE}{{\mathcal E}}

\newcommand{\cO}{{\mathcal O}}
\newcommand{\cG}{{\mathcal G}}
\newcommand{\calL}{{\mathcal L}}
\newcommand{\cM}{{\mathcal M}}
\newcommand{\cI}{{\mathcal I}}

\newcommand{\ati}{{{\mathcal D}_{\cE}}}

\newcommand{\Image}{\operatorname{Im}}
\newcommand{\Kernel}{\operatorname{Ker}}

\newcommand{\gr}{\operatorname{gr}}
\newcommand{\rk}{\operatorname{rk}}

\newcommand{\C}{{\mathbb C}}

\newcommand{\A}{{\mathscr  A}}

\newcommand{\qee}{\mbox{\hspace{0.2mm}}\hfill$\triangle$}

\newcommand{\Der}{\operatorname{Der}}
\newcommand{\iso}{\stackrel{\sim}{\to}}

\newcommand{\cEnd}{{\mathcal E}nd}
\newcommand{\ob}{\text{\bf ob}}

\newcommand{\Out}{\operatorname{Out}}
\newcommand{\ad}{\text{ad}}

\begin{document}
 \begin{flushright}
 SISSA  Preprint 06/2014/mate \\
 \end{flushright}
\bigskip\bigskip

\begin{center}
	{\large \bf NONABELIAN HOLOMORPHIC  LIE\\[5pt] ALGEBROID EXTENSIONS } \\[30pt]
{\sc Ugo Bruzzo,$^{\S\P}$ Igor Mencattini$^\ddag$,\\ Vladimir N. Rubtsov$^{\star\sharp}$
and Pietro Tortella$^\ddag$} \\[10pt]
\small
$^\S$ Scuola Internazionale Superiore di Studi Avanzati (SISSA),  \\ Via Bonomea 265, 34136 Trieste, Italy  \\[3pt]
$^\P$ Istituto Nazionale di Fisica Nucleare (INFN), Sezione di Trieste \\[3pt]
$^\ddag$ Instituto de Ci\^encias Matem\'aticas e de Computa\c c\~ao - USP, \\
Avenida Trabalhador S\~ao-Carlense,  400 - Centro \\
CEP: 13566-590 - S\~ao Carlos - SP, Brazil  \\[3pt]
$^\star$ Universit\'e d'Angers, D\'epartement de Math\'ematiques,  \\ UFR Sciences, LAREMA,  UMR 6093 du CNRS, \\
2 bd.~Lavoisier, 49045 Angers Cedex 01, France \\[3pt]
$^\sharp$ ITEP Theoretical Division, 25 Bol.~Tcheremushkinskaya,  \\ 117259, Moscow, Russia\\[10pt]
E-mail: {\tt bruzzo@sissa.it}, {\tt igorre@icmc.usp.br}, \\ {\tt Volodya.Roubtsov@univ-angers.fr}, {\tt Pietro.Tortella@icmc.usp.br}
\end{center} 
\vfill
\par
\noindent
\begin{quote}
\small  {\sc Abstract.}   We classify nonabelian extensions of Lie algebroids in the holomorphic category. Moreover we study a spectral sequence associated to any such extension. This spectral sequence generalizes the Hochschild-Serre spectral sequence for Lie algebras to the holomorphic Lie algebroid setting.  As an application, we show that the hypercohomology of the Atiyah algebroid of a line bundle has a natural Hodge structure.
\end{quote}
\vfill
\noindent\parbox{.75\textwidth}{\hrulefill}\par
\noindent\begin{minipage}[c]{\textwidth} \footnotesize 
 \noindent   
Date: Revised 20 November 2014  \par\noindent
The authors gratefully acknowledge financial support and hospitality during  visits to
Universit\'e d'Angers, {\sc sissa} and {\sc usp}  S\~ao Carlos. Support for this work was provided by {\sc prin} ``Geometria delle variet\`a algebriche e dei loro spazi di moduli,''    {\sc in}d{\sc am-gnsaga,} the {\sc infn} project {\sc pi}{\tiny 14} ``Nonperturbative dynamics of gauge theories",   the {\sc geanpyl-ii}  Angers-{\sc sissa} project
``Generalized Lie algebroid strucures,'' and  {\sc    diadem} and by {\sc fapesp} (S\~ao Paulo State Research Foundation)
through  the grants  2011/17593-9 and 2012/07867-7.
\end{minipage}
\newpage

\setcounter{tocdepth}{1}
{\small\tableofcontents}

\section{Introduction} 

In this paper we study   nonabelian extensions of Lie algebroids in the holomorphic category. The same theory applies to Lie algebroids over schemes over a field of characteristic zero. Given a Lie algebroid $\cB$ and a totally intransitive Lie algebroid $\calL$ on a complex manifold $X$, we prove the existence of a cohomology class which obstructs the extension of $\cB$ by $\calL$. We also show that the set of isomorphism classes of   extensions is a torsor over a suitable cohomology group. 

The space of global sections of a Lie algebroid $\cA$ over a complex manifold $X$ has a natural structure of Lie-Rinehart algebra.
If the manifold $X$ is Stein, the  Lie algebroid $\cA$ can be reconstructed from the space of its global sections with its Lie-Rinehart algebra structure.   Thus, locally, the extension problem for Lie algebroids reduces to the extension problem for Lie-Rinehart algebras, which, albeit in a different language, was studied in \cite{mackenzie}.

So, our analysis of the extension problem for  holomorphic Lie algebroid starts  in Section \ref{extensions} of this paper  by studying in some detail the extension problem for Lie-Rinehart algebras. As soon as we move to the global problem, we need to use hypercohomology and the theory of derived functors. 
For this reason, and  for the reader's convenience, we basically rework the theory from scratch, instead of relying too much on the existing literature. Both the definition of the obstruction class, and the classification of the extensions, are obtained by constructing explicit \v Cech cocyles.

In some sense, the theory of extensions of Lie algebroids, as we approach it, is one more step in the development of the
 theory of nonabelian extensions, which was initiated by O.\ Schreier almost a century ago, with the theory of nonabelian group extensions \cite{Schr1,Schr2}. The theory was reworked and recast in a simpler and more modern form by Eilenberg and MacLane in the 1940s \cite{EML1,EML2}. The same kind of machinery solves the nonabelian Lie algebra extension problem, see \cite{Hoch,Mori,Shuk,mackenzie}; a nice rendering of the theory,
with an explicit construction of all the cocycles involved, is given in \cite{TheThreeGuys}. The theory extends to Lie algebroids along the same lines.

A generalization of our result about the nonabelian extensions of Lie algebroids to Courant algebroids was given in \cite{Davide}.

In a second part of the paper, given an exact sequence of Lie algebroids $$ 0\to \calL \to \cA \to \cB\to 0,$$  we note that
the de Rham complex $\Omega_\cA^\bullet$ of $\cA$ has a natural filtration. This induces 
 a spectral sequence   that converges to the hypercohomology $\mathbb H(X,\Omega_\cA^\bullet)$.  This spectral sequence generalizes the one defined by Hochschild and Serre for any pair $(\mathfrak g,\mathfrak h)$, where $\mathfrak g$ is a Lie algebra and $\mathfrak h$ is a Lie subalgebra of $\mathfrak g$. 
We give a general form for the $E_1$ term and, by means of an explicit computation,   find a description of the $d_1$ differential. 
Some of these results for the algebraic setting (i.e., for Lie-Rinehart algebras) can be found in   \cite{Rub-thesis,Roub80}.

In Section \ref{linebundles} we study in some detail the spectral sequence in the case of the Atiyah algebroid of a line bundle $\mathcal M$. We give a very explicit realization of the differential $d_1$, and show that when $X$ is compact K\"ahler, as a consequence of the Hodge decomposition, the spectral sequence degenerates at the second step (in particular, the differential $d_2$ is zero). This also allows one to equip the hypercohomology of the Atiyah algebroid of $\mathcal M$ with a  Hodge structure.

{\bf Acknowledgements.} We thank Paul Bressler, Tony Pantev and Jean-Claude Thomas  for useful discussions. We acknowledge hospitality and support by the Max-Planck-Institut f\"ur Mathematik in Bonn, where part of this work was carried out.
 
\bigskip\section{Generalities on Lie algebroids }\label{gen}
     
Let $X$ be a complex manifold.\footnote{Let us stress once more that, while  in this paper we  work only in the holomorphic category, the techniques used here can be also applied to  the algebraic one, i.e., when $X$ is a smooth scheme over an algebraically closed field of characteristic $0$, and $\cA$, $\calL$ and $\cB$ are algebraic Lie algebroids.} We shall denote by $\cO_X$ the sheaf of holomorphic functions on $X$ and by $\mathbb C_X$ the constant sheaf on $X$ with fibre $\mathbb C$; with $\Theta_X$ we shall denote the tangent sheaf of $X$.  
By ``vector bundle'' we shall mean a finitely generated locally free $\cO_X$-module.

A holomorphic Lie algebroid $\cA$ on $X$ is a coherent $\cO_X$-module $\cA$ equipped with a Lie algebroid structure, that is, a morphism of $\cO_X$-modules $a\colon \cA \to\Theta_X$, called the {\em anchor} of $\cA$, and a Lie bracket defined on the sections of $ \cA$     satisfying the  Leibniz rule
\begin{equation}\label{leibniz} [s,ft] = f[s,t] +a(s)(f)\,t \end{equation}
for all sections $s,t$ of $\cA$ and $f$ of $\cO_X$.
This condition implies that the anchor is a morphism of sheaves of $\mathbb C_X$-Lie algebras, considering $\Theta_X$  as a sheaf of $\mathbb C_X$-Lie algebras with respect to the commutator of vector fields. 

A morphism $(\cA,a)\to (\cA',a')$ of Lie algebroids defined over the same manifold $X$ is a morphism of $\cO_X$-modules $f\colon\cA\to\cA'$, which is compatible with the brackets defined in $\cA$ and in $\cA'$, and such that $a'\circ f=a$. 

Let us introduce the {\em de Rham complex}  of $\cA$, which is a  sheaf of
differential graded algebras. This is  $\Omega_{ \cA}^\bullet=\Lambda^\bullet_{\cO_X}\cA^\ast$ as a sheaf of $\cO_X$-modules, with a product  given by the wedge product $\bullet \wedge \bullet$, and   differential  $d_\cA\colon \Omega_{ \cA}^\bullet\to \Omega_{ \cA}^{\bullet+1}$   defined by the  formula
  \begin{eqnarray*}\label{diff}
(d_\cA\xi)(s_1,\dots,s_{p+1}) &=& 
\sum_{i=1}^{p+1}(-1)^{i-1}a(s_i)(\xi(s_1,\dots,\hat s_i,
\dots,s_{p+1})) \\ & + & \sum_{i<j}(-1)^{i+j}
\xi([s_i,s_j],\dots,\hat s_i,\dots,\hat s_j,\dots,s_{p+1})
\end{eqnarray*}
  for   $s_1,\dots,s_{p+1}$ sections of $\cA$, and $\xi$ a section of $\Omega_{ \cA}^p$.
  The hypercohomology of the complex $(\Omega_{ \A}^\bullet,d_\cA)$ is called the {\em holomorphic Lie algebroid cohomology} of $\cA$.
When $\cA$ is locally free as an $\cO_X$-module, this cohomology is  isomorphic to the Lie algebroid cohomology of the smooth complex Lie algebroid $A$ obtained by matching (in the sense of \cite{Lu97,Mokri,GSX}) the holomorphic Lie algebroid $\cA$ with the anti-holomorphic tangent bundle $T_X^{0,1}$ \cite{GSX,BR-cohom}. Since $\rk A = \rk  \A + \dim X$,  the hypercohomology of the complex $\Omega_{ \cA}^\bullet$ vanishes in degree higher than $\rk  \A + \dim X$. 
        
\begin{ex}[The Atiyah algebroid]
A fundamental example of Lie algebroid is $\ati$, the {\em Atiyah algebroid} associated with a coherent $\cO_X$-module $\cE$.
This is defined as the sheaf of first order differential operators   on $\cE$   having scalar symbol. $\ati$ sits inside
the   exact sequence of sheaves of $\cO_X$-modules
\begin{equation} \label{atiyah} 0 \to  \cEnd_{\mathcal O_X}(\cE) \to \ati \xrightarrow{\sigma } \Theta_X ,
\end{equation} where $\sigma$, called {\it the symbol map}, plays the role of the anchor. The bracket is given by the commutator of differential operators.
Note that, when $\cE$ is locally free, the anchor is surjective and $\ati$ is an extension of $\Theta_X$ by $\cEnd_{\cO_X} (\cE)$.
\end{ex}

\begin{defin}
Given $\cM$ a coherent $\cO_X$-module, a {\em $\cA$-connection} on $\cM$ is a morphism of $\cO_X$-modules $\alpha: \cA \to \mathcal{D}_{\!\!\cM}$ which commutes with the anchors of $\cA$ and $ \mathcal{D}_{\!\!\cM}$.
Moreover, an  $\cA$-connection $\alpha$ is said to be {\em flat}, and the pair $(\cM,\alpha)$ is called a {\em representation} of $\cA$ (also called an $\cA$-module), if $\alpha$ is a morphism of Lie algebroids. 
\end{defin}

Given a $\cA$-connection $(\cM,\alpha)$, one can introduce the twisted modules $\Omega_{ \cA}^\bullet(\cM)$, where $\Omega_{\cA}^k(\cM) = \Lambda^k\cA^\ast\otimes \cM$, with a differential

 \begin{eqnarray*}\label{diff}
(d_\alpha \xi)(s_1,\dots,s_{p+1}) &=& 
\sum_{i=1}^{p+1}(-1)^{i-1}\alpha (s_i)(\xi(s_1,\dots,\hat s_i,
\dots,s_{p+1})) \\ & + & \sum_{i<j}(-1)^{i+j}
\xi([s_i,s_j],\dots,\hat s_i,\dots,\hat s_j,\dots,s_{p+1}).
\end{eqnarray*}
We shall call the sections of this twisted module   $\cM$-valued $\cA$-forms.

If the connection is flat, $\Omega_\cA^\bullet(\cM)$ is a complex of $\cO_X$-modules, and its hypercohomology is called the cohomology of $\cA$ with values in $(\cM,\alpha)$.
If $\alpha$ is not flat, one can introduce its {\em curvature} $F_\alpha$ as the defect of $\alpha$ to be a representation, i.e.,
\[
F_\alpha(s_1,s_2) = [\alpha(s_1),\alpha(s_2)] - \alpha([s_1,s_2]).
\]
Observe that $F_\alpha$ is $\cO_X$-bilinear and   satisfies the {\it Bianchi identity} $d_\alpha F_\alpha= 0$. The curvature measures also the defect of $(\Omega^\bullet_\cA(\cM), d_\cA)$ to be a complex. In fact one can check that $d_\cA^2 = F_\alpha \smile \bullet$, where for $\xi \in \Omega^p_\cA(\cM)$:    

\begin{equation} \label{curvature_1}
\big(F_\alpha \smile \xi \big) \  (s_1,\ldots, s_{p+2}) = \sum_{i<j} (-1)^{i+j} F_\alpha(s_i,s_j) \big( \xi(s_1,\ldots,\hat{s}_i,\ldots,\hat{s}_j,\ldots, s_{p+2}) \big).
\end{equation}
Note that $F_\alpha \in \Omega_\cA^2(\End_{\cO_X}(\cM))$, and the last equation is given by the cup product coming from the     evaluation morphism  $\cEnd_{\cO_X}(\cM)\otimes_{\cO_X}\cM\to\cM$.

\subsection{Totally intransitive Lie algebroids}

In this subsection we shall collect a certain number of definitions that will be used in what follows.

\begin{defin}
A Lie algebroid $\calL$ whose anchor is zero is said {\em totally intransitive}. 
\end{defin}

\begin{remark} The notion of totally intransitive Lie algebroid is equivalent to that of a sheaf of $\cO_X$-Lie algebras.  The structure of  an $\cO_X$-Lie algebra sheaf  induces a $\mathbb C$-Lie algebra structure on $\calL(x)$, the fibre of $\calL$ at $x\in X$, which, in general, may vary from point to point. When this does not happen, i.e., when there exists a $\mathbb C$-Lie algebra $\mathfrak g$ and local trivializations $\mathcal L_{|U} \iso \cO_U \otimes \mathfrak g$ that are isomorphisms of Lie algebras, $\calL$ is called a {\em Lie algebra bundle}, cf.\ \cite{mackenzie}.
\end{remark}

Let $\calL$ be a  totally intransitive Lie algebroid.

\begin{defin}
$\Der_{\cO_X} (\calL)$ is the the sheaf of the sections $\phi\in \cEnd_{\cO_X}(\calL)$ such that
\begin{equation} \label{derivations}
\phi([l,l']) = [\phi(l), l'] + [l,\phi(l')]. 
\end{equation}
for any $l,l' \in \calL$. 
\end{defin}

Let $\ad\colon \calL \to \Der_{\cO_X}(\calL)$ be the morphism of $\cO_X$-modules   that with every $l\in \calL$ associates $\ad_l: l'\mapsto [l,l']$. 
\begin{defin}
 We define
\[
\Out_{\cO_X}(\calL)= \Der_{\cO_X}(\calL)/\ad(\calL).
\] 
\end{defin}

Note that, since $\ad(\calL)$ is an ideal in $\Der_{\cO_X}(\calL)$, the sheaf  $\Out_{\cO_X}(\calL)$    is a totally intransitive Lie algebroid. Moreover, as $\calL$ and its Atiyah algebroid $\calD_{\!\! \calL}$ are $\cO_X$-modules,
we can define:
\begin{defin}
$\Der_{\calD}(\calL)$ is the sheaf of the sections of $\calD_{\!\!\calL}$ that satisfy the equation \eqref{derivations}. 
\end{defin}

Note   that  $\Der_{\calD}(\calL)$ is a sub-Lie algebroid of $\calD_{\!\!\calL}$, and
the natural inclusion $\Der_{\cO_X}(\calL) \hookrightarrow \Der_{\calD}(\calL)$ induces an inclusion $\ad(\calL) \hookrightarrow \Der_{\calD}(\calL)$. Under this map $\ad(\calL)$ is an ideal in $\Der_{\calD}(\calL)$.

Finally, we define:

\begin{defin}
$\Out_{\calD}(\calL)$ is the quotient of $\Der_{\calD}(\calL)$ by $\ad(\calL)$.
\end{defin}

Note that $\Out_{\calD}(\calL)$ has a natural structure of Lie algebroid.
In general,
for a Lie algebroid $\cA$ with a nontrivial anchor,    $\Der_{\cO_X}(\cA)$, $\Der_{\calD}(\cA)$, and $\ad(\cA)$ are badly behaved. In fact, because of the Leibniz identity, $\Der_{\cO_X}(\cA)$ and $\Der_{\calD}(\cA)$) are not     sub-$\cO_X$-modules of the corresponding vector bundles, and the adjoint map does not yield a representation of $\cA$ on itself (cf.\  \cite{up_to_homotopy}).

\begin{remark} \label{bracket}
In the next sections we shall be dealing with a Lie algebroid $\cA$ and a totally intransitive Lie algebroid $\calL$ endowed with a $\cA$-connection $\alpha: \cA \rightarrow \Der_{\calD}(\calL)$. 
We shall also consider $\Omega_{\cA}^\bullet (\calL)$, the sheaf of $\cO_X$-modules of   $\calL$-valued $\cA$-forms. The bracket on $\calL$ induces a bracket $[\cdot,\cdot]: \Omega_{\cA}^p(\calL) \otimes \Omega^q_{\cA}(\calL) \to \Omega^{p+q}_\cA(\calL)$ by means of a cup-product construction. The explicit formula is
\[
[\xi,\eta](b_1,\ldots,b_{p+q}) =\sum_{\sigma\in \Sigma_{p,q}} (-1)^\sigma [\xi(b_{\sigma(1)},\ldots,b_{\sigma(p)}) , \eta(b_{\sigma(p+1)},\ldots,b_{\sigma(p+q)}) ]\ ,
\]
where $\Sigma_{p,q}$ denotes the set of $(p,q)$-shuffles, i.e., the set of permutations $\sigma \in \Sigma_{p+q}$ such that $\sigma(1)<\sigma(2)<\cdots <\sigma(p)$ and $\sigma(p+1)<\sigma(p+2)<\cdots< \sigma(p+q)$. 
This bracket is graded skew-symmetric, i.e.,
\[
[\xi,\eta] = (-1)^{pq+1} [\eta,\xi] \quad\mbox{for all}\quad \xi\in\Omega^p_{\cA}(\calL)\quad\text{and}\quad\eta\in\Omega^q_{\cA}(\calL),
\]
and   satisfies a graded Jacobi identity. 
\end{remark}

\bigskip 

\section{Extensions of Lie algebroids}\label{extensions}

In this section we compute the obstruction to the existence of an extension of a holomorphic Lie algebroid by a totally intransitive Lie algebroid, and     classify such extensions when the obstruction vanishes. 
In the smooth category, this problem has been extensively treated in the literature. For a comprehensive discussion we refer the reader to \cite{mackenzie}, see also \cite{TheThreeGuys,BKTV09,brahic}.
Moreover,  the problem of the abelian extensions of   holomorphic Lie algebroids was solved in \cite{BB} for $\cB= \Theta_X$, and in \cite{PietroCEJM}   for a general $\cB$ (with the notation of equation \eqref{ext}).

\subsection{The general problem}

Let
\begin{equation} \label{ext}  
0 \to \calL \to \cA \to \cB \to 0. 
\end{equation}
be an extension of Lie algebroids, i.e., $\calL, \cA$ and $\cB$ are holomorphic Lie algebroids over a complex manifold $X$, the map $\calL\to \cA$ is injective, the map $\cA \to \cB$ is surjective and their composition is zero.
It follows from the definition of Lie algebroid morphism that the anchor of $\calL$ is zero, i.e.,  $\calL$ is a totally intransitive Lie algebroid.

\begin{assumption}
In what follows, unless differently stated, $\cB$ will be a locally free $\cO_X$-module.
\end{assumption}

Let an extension as \eqref{ext} be given. Then   one defines a morphism of Lie algebroids
\begin{equation} \label{alpha} 
\bar\alpha \colon \cB \to \Out_\calD(\calL) 
\end{equation}
 by letting 
\[
\bar\alpha (s) (\ell) = [s',\ell]_{\cA},
\]
with $s'$ any pre-image of $s$ in $\cA$.
The morphism $ \bar\alpha \colon \cB \to \Out_\calD(\calL) $ defines a representation of $\cB$ on the centre $Z(\calL)$ of $\calL$, so that one can consider the hypercohomology
$\mathbb H \big( X;\Omega^\bullet_{\cB} \big( Z(\calL)),\bar\alpha\big)$.
Note that the action of $\cB$ on the centre $Z(\calL)$ is defined once a morphism $\bar\alpha$ has been chosen. Thus  both the $\cB$-module structure on $Z(\calL)$ and the corresponding (hyper)cohomology groups depend on this choice. In spite of this,  we       adopt a notation that does not stress this dependence. For example, to denote the hypercohomology groups above mentioned, we shall prefer the simpler notation $\mathbb H \big(X;\Omega^\bullet_{\cB} \big( Z(\calL)\big)\big)$ to the more precise   $\mathbb H \big( X;\Omega^\bullet_{\cB} \big( Z(\calL)),\bar{\alpha}\big)$.

We can reformulate the extension problem more precisely as follows:
{\em given two Lie algebroids $\cB$ and $\calL$ on $X$, with $\calL$ totally intransitive, and given 
\[
\bar\alpha \colon \cB \to \Out_\calD(\calL)\]
as above,
\begin{enumerate}
\item does there exist an extension as in \eqref{ext} inducing $\bar\alpha$? 
\item If such extensions exist, how are they classified?
\end{enumerate}
}

Some notational remarks are now in order.
\begin{enumerate}
\item The pair $(\calL,\bar{\alpha})$ is sometimes called a coupling of $\cB$, see \cite{mackenzie}.
\item Two extensions $\cA$, $\cA'$ of $\cB$ by $\calL$ are called {\it equivalent} if there exists a Lie algebroid morphism $\cA\to\A'$ that makes the following diagram commutative:
\[
\xymatrix{
0 \ar[r]  & \calL \ar[r] \ar@{=}[d] & \cA \ar[r] \ar[d] & \cB \ar[r] \ar@{=}[d] & 0 \\
0 \ar[r] & \calL \ar[r] & \cA'\ar[r] & \cB \ar[r] & 0\ .
}
\]
\item Given a complex of sheaves $\mathcal F^\bullet$ indexed by $\mathbb{Z}$ and given $a \in \mathbb{Z}$, we denote by $\tau^{\geq a} \mathcal F^\bullet$ the truncated complex, obtained by replacing the sheaves $\mathcal F^b$ by $0$ for $b<a$, i.e.,
\[
\tau^{\geq a} \mathcal{F} = \cdots \to 0 \to 0 \to \mathcal{F}^a \to \mathcal F^{a+1} \to \cdots 
\]

\end{enumerate}

The main result of this section is contained in the following theorem.

\begin{thm}\label{extthm}
 There is a cohomology class $\ob(\bar\alpha) \in \mathbb H^3 \big(X ; \tau^{\geq 1} \Omega^\bullet_{\cB} ( Z(\calL)) \big)$
such that:
\begin{enumerate}
\item there exists  a (holomorphic) Lie algebroid extension of $\cB$ by $\calL$ inducing the map $\bar{\alpha}$ if and only if $\ob(\bar\alpha)  = 0$.
\item When $\ob(\bar\alpha)  = 0$, the set of isomorphism classes of extensions of $\cB$ by  $\calL$ inducing $\bar\alpha$ is a torsor over  $\mathbb H^2 \big(X; \tau^{\geq 1} \Omega^\bullet_{\cB} ( Z(\calL)) \big)$.
\end{enumerate}
\end{thm}

To simplify the exposition, we shall split the proof of the Theorem \ref{extthm} into two subsections. In the first  we shall assume that $X$ is a Stein manifold, while in the second   $X$ will be any complex manifold.

\subsection{The local case} 
\label{LRalgebras}

If $X$ is a Stein manifold, instead of studying   the extensions of $\cB$ by $\calL$ (two coherent $\cO_X$-modules), we can study the problem of the extensions of $B$ by $L$ (the spaces of   global sections of $\cB$ and  $\calL$).
Indeed, on a Stein manifold any coherent $\cO_X$-module is generated by its global sections, and  thus the problem of the extensions of coherent $\cO_X$-modules reduces to a purely algebraic one. Since in this context the terminology of the Lie-Rinehart algebras is commonly used, we shall start recalling the relevant definitions for this class of algebras; see also \cite{LR} for some history about this subject. The results presented in this subsection are an adaptation to the Lie-Rinehart terminology of the results about the extensions of   smooth real Lie algebroids, for which we refer the reader to   Chapter 7 of \cite{mackenzie}. 

\begin{defin}
Let $R$ be a unital commutative algebra over a field $k$. A {\em $(k,R)$-Lie-Rinehart algebra} is a pair $(B,b)$, where $B$ is a $k$-Lie algebra and $b\colon B \to \operatorname{Der}_k(R)$ a representation of $B$ in $\operatorname{Der}_k(R)$ that satisfies the Leibniz identity \eqref{leibniz}.
\end{defin}

\begin{remark}  
(i) Let $M$ be a smooth real manifold, and   $E$ a Lie algebroid   on it. The space of global sections $\Gamma(E)$ of $E$ is a $(\mathbb R, C^\infty(M))$-Lie-Rinehart algebra. Viceversa, any $(\mathbb R, C^\infty(M))$-Lie-Rinehart algebra, which is projective as a $C^\infty(M)$-module, is the space of sections of a smooth Lie algebroid.

(ii)  Given a holomorphic Lie algebroid $\cA$   over $X$ and   an open subset $U\subseteq X$, the vector space $\cA(U)$ of   sections of $\cA$ over $U$ has a natural structure of   $\big(\mathbb{C},\cO_X(U)\big)$-Lie-Rinehart algebra.
\end{remark}

As for Lie algebroids,   for   Lie-Rinehart algebras it is possible to define a cohomology theory, which reduces to the Lie algebra Chevalley-Eilenberg cohomology when $b=0$. In the same vein, all   notions introduced in Section 2 have an analogue in the theory of the Lie-Rinehart algebras. For example, given an $R$-module $M$, it is possible to define the corresponding notion of Atiyah algebra $\mathcal D_M$,  of   $B$-connection, and of  representation of $B$. Moreover, given a $B$-connection $(M,\alpha)$, one can   introduce the twisted modules $\Omega_B^p(M)$ and the corresponding differential $d_\alpha$. Finally, if $M$ is a representation of $B$, we shall denote by $H^k(B;M)$ the $k$th cohomology of the complex $\Omega_B^\bullet(M)$. 
      
Before   stating the main result of this subsection we introduce the following notion. 

\begin{defin}\label{def:liftpair}  
A {\em lifting pair} of the coupling $(L,\bar\alpha)$ is the data of a pair $(\alpha,\rho)$, where:
\begin{enumerate}
\item $\alpha: B \to \Der_\calD(L)$ is a lifting of $\bar{\alpha}$ and
\item $\rho \in \Omega^2_B(L)$ satisfies $\ad_\rho = F_\alpha$.
\end{enumerate}
\end{defin}

Note that lifting pairs always exist: any lifting $\alpha: B \rightarrow \Der_\calD L$ of $\bar{\alpha}$ is a $B$-connection on $L$, and since $\bar \alpha$ is a morphism of Lie algebroids, there exists $\rho \in \Omega^2_B(L)$ such that $F_\alpha(b_1,b_2) = \ad_{\rho(b_1,b_2)}$.

\begin{thm} \label{thm_local}
Let $R$ be a unital commutative $k$-algebra, and   $(B,b)$   a $(k,R)$-Lie-Rinehart algebra, where $B$ is a projective $R$-module. Let $(\bar{\alpha},L)$ be a coupling of $B$, where $L$ is a $R$-Lie-algebra, and $\bar{\alpha}:B\to \Out_\calD L$ a morphism of Lie-Rinehart algebras.
\begin{enumerate}
\item There exists a class $\ob(\bar{\alpha}) \in H^3(B;Z(L))$ such that $\ob(\bar{\alpha}) = 0$ if and only if there exists an extension of $B$ by $L$ inducing $\bar{\alpha}$;
\item  when the class $\ob(\bar{\alpha})$ is zero, the set of extensions of $B$ by $L$ has the structure of a torsor over $H^2(B; Z(L))$.
\end{enumerate}
\end{thm}
\begin{proof}
We begin by constructing the obstruction class. Let $(\alpha,\rho)$ be a lifting pair of $\bar{\alpha}$ and define $\lambda_{(\alpha,\rho)} = d_\alpha \rho$. One checks that  $\lambda_{(\alpha,\rho)}$ takes values in $Z(L)$, so that it defines an element in $\Omega^3_B\big( Z(L) \big)$. Since $d_{\bar{\alpha}} \lambda = d^2_\alpha \rho$, from Equation \eqref{curvature_1} and Remark \ref{bracket} we obtain that
$d_{\bar{\alpha}} \lambda = [\rho,\rho] $
vanishes, as the bracket is skew-symmetric on two forms. The obstruction class $\ob(\bar{\alpha})$   is then the cohomology class $[\lambda_{(\alpha,\rho)}]\in H^3(B;Z(L))$. 
We need to show that the definition of this class does not depend on the choice of the lifting pair, i.e., we need to prove that if $(\alpha',\rho')$ is another lifting pair, then the corresponding $\lambda_{(\alpha',\rho')}$ is cohomologous to $\lambda_{(\alpha,\rho)}$. 
To this end, first assume that $\alpha' = \alpha$. In this case  $\ad_{\rho'-\rho} = F_\alpha - F_\alpha = 0$, so that $\rho'-\rho$ takes values in $Z(L)$, which implies  
\[
\lambda_{(\alpha,\rho')} - \lambda_{(\alpha,\rho)} = d_{\bar{\alpha}} (\rho'-\rho),
\]
proving that $\lambda_{(\alpha,\rho)}$ and $\lambda_{(\alpha,\rho')}$ are cohomologous.
Let now $(\alpha',\rho')$ be any other lifting pair of $\bar{\alpha}$. Then there exists $\phi \in \Omega^1_B(L)$ such that $\alpha'-\alpha = \ad_\phi$. The following lemma proves that the class $\ob(\bar\alpha)$ is independent of the choice of the lifting pair.
\begin{lemma} \label{change_lifting_pairs}
Let $(\alpha,\rho)$ be a lifting pair of $\bar{\alpha}$, and $\phi \in \Omega^1_B(L)$.
Then $$(\alpha + \ad_\phi , \rho + d_\alpha \phi + \frac{1}{2} [\phi,\phi])$$ is another lifting pair of $\bar{\alpha}$, and we have 
\[
\lambda_{(\alpha + \ad_\phi , \rho + d_\alpha \phi + \frac{1}{2} [\phi,\phi])} = \lambda_{(\alpha,\rho)}.
\]
\end{lemma}

This   lemma is proved  by   direct computation,   using the fact that  $d_\alpha$  is a graded derivation of the bracket $[\cdot,\cdot]$ defined on $\Omega^\bullet_B(L)$, together with the  result established by the following lemma. 
\begin{lemma} \label{lem_change_d}
Let $\alpha, \alpha'$ be two liftings of $\bar{\alpha}$ such that $\alpha'-\alpha = \ad_{\phi}$ for some $\phi:B \to L$.
For every $\eta \in \Omega^p_B(L)$, the corresponding differentials $d_{\alpha},d_{\alpha'}$ satisfy the   formula
\begin{equation}\label{change_d}
d_{\alpha'} \eta - d_\alpha \eta = [\phi,\eta].
\end{equation}
\end{lemma}

Now, given a lifting pair $(\alpha,\rho)$ of $(L,\bar{\alpha})$, and after defining
\begin{equation}\label{bracket_split}
[(b,l),(b',l')]_{\alpha,\rho} = ([b,b'] , [l,l'] + \alpha(b)(l')- \alpha(b')(l) +\rho(b,b')),\,\forall b,b' \in B,\,  l,l' \in L,
\end{equation}
we can   state the following lemma, whose proof is a straightforward computation.

\begin{lemma} \label{lemma_b} Given a lifting pair $(\alpha,\rho)$:
\begin{enumerate}
\item[(1)] the bracket $[\cdot, \cdot]_{\alpha,\rho}$ equips $B\oplus L$ with a skew-symmetric bracket which is compatible with $\bar{\alpha}$, with the canonical maps $L \to B\oplus L$ and $B \oplus L \to B$, and with the brackets defined on $L$ and  $B$; 
\item[(2)] the bracket $[\cdot, \cdot]_{\alpha,\rho}$ defined in (\ref{bracket_split}) satisfies the Jacobi identity if and only if  $\lambda_{(\alpha,\rho)} = 0$. When this happens, $(B\oplus L , [\cdot, \cdot]_{\alpha,\rho})$ is a Lie-Rinehart algebra extension of $B$ by $L$, inducing the coupling $\bar{\alpha}$;
\item[(3)] two lifting pairs $(\alpha,\rho)$ and $(\alpha',\rho')$ define equivalent extensions if and only if there exists $\eta \in \Omega^1_B(L)$ such that $\alpha'-\alpha = \ad_\eta$, and 
\[
\rho'-\rho - d_\alpha \eta +\frac{1}{2} [\eta,\eta]= d_{\bar{\alpha}} \beta
\] 
for some $\beta \in \Omega_B^1(Z(L))$.
\end{enumerate}
\end{lemma}

We can now conclude the proof of   Theorem \ref{thm_local}. 

 (i)  Assume that an extension as \eqref{ext} is given. Choose a splitting $s: B \to A$, which exists as $B$ is assumed to be 
a projective $R$-module. Define $\alpha_s: B \to \Der_\calD L$ via the formula $\alpha(b)(l) = [s(b),l]$, and $\rho_s \in \Omega^2_B(L)$ via $\rho_s(b_1,b_2) = [s(b_1), s(b_2)] - s([b_1,b_2])$. 
Then $(\alpha_s,\rho_s)$ is a lifting pair of $\bar{\alpha}$, so that $\lambda_s = \lambda_{(\alpha_s,\rho_s)}$ is a representative of $\ob(\bar{\alpha})$. The splitting $s$ induces also an isomorphism $A \stackrel{s}{\to} B \oplus L$. Under this isomorphism, one obtains a bracket on $B\oplus L$ from the bracket on $A$. This bracket coincides with $[\cdot,\cdot]_{\alpha_s,\rho_s}$. Then $\ob(\bar{\alpha}) = [\lambda_s] = 0$ by Lemma \ref{lemma_b}.

Conversely, assume that $\ob(\bar{\alpha}) = 0$. Choose any lifting pair $(\alpha,\rho)$ of $\bar{\alpha}$. Since the cohomology class of $\lambda_{(\alpha,\rho)}$ is zero, there exists $w\in \Omega^2_B(Z(L))$ such that $d_\alpha \rho = d_{\bar \alpha} w$. Now, since $w$ takes values in the centre of $L$, also $(\alpha, \rho-w)$ is a lifting pair.
On the other hand, from the definition of $w$ we have $d_\alpha(\rho - w) = 0$, so by Lemma \ref{lemma_b} the bracket $[\cdot,\cdot]_{\alpha,\rho-w}$ on $B\oplus L$ is an extension of $B$ by $L$. 

(ii)  To conclude the proof one needs to define an action of $H^2(B; Z(L))$ on the set of equivalence classes of the extensions. Let an extension be given, choose a splitting $s:B \to A$, and construct a lifting pair $(\alpha_s,\rho_s)$ as before. Let $\gamma \in \Omega^2_{B}(Z(L))$, with $d_{\bar \alpha} \gamma = 0$, be a representative of a class in $H^2(B; Z(L))$. Then $[\cdot,\cdot]_{\alpha_s, \rho_s + \gamma}$ defines a new extension, which by Lemma \ref{lemma_b} is equivalent to $[\cdot,\cdot]_{\alpha_s,\rho_s}$ if and only if $\gamma = d_{\bar{\alpha}} \beta$.

This action defines a torsor structure because if one is given two extensions $A$ and $ A'$,  by choosing two splittings $s:B\to A$ and $s':B\to A'$, one obtains two lifting pairs $(\alpha_s,\rho_s)$ and $(\alpha_{s'},\rho_{s'})$. One checks that, for $\phi\in \Omega^1_B(L)$ such that $\alpha'-\alpha= \ad_\phi$, the formula
\begin{equation}\label{rho_phi}
\gamma = \rho'-\rho - d_\alpha \phi - \frac{1}{2}[\phi,\phi]
\end{equation}
defines a $d_{\bar{\alpha}}$-closed element $\gamma\in\Omega_B^2(Z(L))$, which  singles out a class in $H^2(B;Z(L))$.
\end{proof}

\subsection{The case of holomorphic Lie algebroids}\label{ss:holoLie}

With the local case in mind, we can now deal with the global situation. Let $\cB$ be a holomorphic Lie algebroid and let $(\calL,\bar\alpha)$ be a coupling of $\cB$, where $\calL$ is a totally intransitive holomorphic Lie algebroid.
We proceed first to the construction of the obstruction class $\ob(\bar{\alpha}) \in \mathbb{H}^3 \big( X; \tau^{\geq 1}\Omega^\bullet_\cB(Z(\calL)) \big)$. Let $\mathfrak U = \{U_i\}$ be a good open cover of $X$ (where by ``good'' we mean that each $U_i$ is a connected Stein manifold and that for each $i_0,\ldots,i_p$ the intersection $U_{i_0\cdots i_p} = \bigcap_a U_{i_a}$, when nonvoid, is a connected Stein manifold). We compute the hypercohomology using   \v{C}ech resolutions. More precisely, $\mathbb{H}^3 \big(X; \tau^{\geq 1}\Omega^\bullet_\cB(Z(\calL)) \big)$ is the degree 3 cohomology of the total complex $T^\bullet$ associated to the double complex $K^{q}_{p} = \check{C}^p (\mathfrak U; \Omega^q_\cB(Z(\calL)))$ for $p\geq 0$ and $q \geq 1$, with differentials $d_1 = \delta $, the \v{C}ech differential, and $d_2 = d_{\bar{\alpha}}$. In this way, a representative of a cohomology class in $\mathbb{H}^3 \big( X; \tau^{\geq 1}\Omega^\bullet_\cB(Z(\calL)) \big)$ will consist of an element of $T^3 = K^{3}_0 \oplus K^{2}_1 \oplus K^{1}_2$. 

Over each $U_i$ we get a coupling for the Lie-Rinehart algebra $\cB(U_i)$, i.e., a pair $(\calL(U_i),\bar\alpha(U_i))$, where $\bar{\alpha}(U_i): \cB(U_i) \to \Out_\calD(\calL)(U_i)$. For each $i$ we can choose a lifting pair $(\alpha_i,\rho_i)$ of $\bar{\alpha}(U_i)$, and on the double intersections we can choose $\phi_{ij} \in \Omega_\cB^1(\calL)(U_{ij})$ such that $\alpha_j - \alpha_i = \ad_{\phi_{ij}}$. This motivates the following:

\begin{defin}
A {\em lifting triple} of $(\calL,\bar{\alpha})$ is a triple $$(\{\alpha_i\},\{\rho_i\},\{\phi_{ij}\})$$ with $\{\alpha_i\} \in \check{C}^0(\mathfrak U ; \Omega^1_\cB(\Der_\calD (\calL) ))$, $\{\rho_i\} \in \check{C}^0(\mathfrak U; \Omega^2_\cB(\calL))$, and $\{\phi_{ij}\} \in \check{C}^1(\mathfrak U; \Omega^1_\cB(\calL))$, such that:
\begin{enumerate}
\item for each $i$ the pair $(\alpha_i,\rho_i)$ is a lifting pair for $\bar{\alpha}_{|U_i}$, 
\item for each $i,j$ on the double intersections $U_{ij}$ we have $\alpha_j - \alpha_i = \ad_{\phi_{ij}}$.
\end{enumerate}
\end{defin}

Now let $(\{\alpha_i\},\{\rho_i\},\{\phi_{ij}\})$ be a lifting triple of $(\calL,\bar{\alpha})$. For each $i$ we can construct $\lambda_i = \lambda_{(\alpha_i,\rho_i)}\in \Omega^3_\cB(Z(\calL)) (U_i)$ as in the previous section. The collection $\{\lambda_i\}$ defines a \v{C}ech cochain in $K^{3}_0$. On the double intersections let
$$
t_{ij} = \rho_j - \rho_i - (d_{\alpha_i} \phi_{ij} + \frac{1}{2}[\phi_{ij},\phi_{ij}]).
$$
These $t_{ij}$ are elements of $\Omega_\cB^2(Z(\calL))(U_{ij})$ (cf. equation \eqref{rho_phi}), and their collection defines the \v{C}ech cochain $\{t_{ij}\} \in K^{2}_1$. Finally, on the triple intersections $U_{ijk}$ let
\[
q_{ijk} = (\delta \phi)_{ijk} = -\phi_{jk} + \phi_{ik} - \phi_{ij}.
\]
By the definition of the $\phi_{ij}$, it follows that $\ad_{q_{ijk}} = 0$, so that one obtains a \v{C}ech cochain $\{q_{ijk}\} \in K^{1}_2$.

\begin{defin}
The triple $(\{\lambda_i\}, \{t_{ij}\}, \{q_{ijk}\})$ defines an element in $T^3$, which will be called the cochain associated to the lifting triple $(\{\alpha_i\},\{\rho_i\},\{\phi_{ij}\})$.
\end{defin}

We can prove the following statement.

\begin{prop} \label{cocycle}
The triple $(\{\lambda_i\}, \{t_{ij}\}, \{q_{ijk}\} ) \in T^3$ is closed under the total differential of $T$, and therefore defines a cohomology class in $\mathbb{H}^3 \big( X; \tau^{\geq 1}\Omega^\bullet_\cB(Z(\calL)) \big)$.
\end{prop}
\begin{proof} The claim is equivalent to the  four equations
\begin{enumerate}
\item $d_{\bar{\alpha}} \lambda_i = 0$;
\item $d_{\bar{\alpha}} t_{ij} = (\delta  \lambda)_{ij}$;
\item $d_{\bar{\alpha}} q_{ijk} = (\delta  t)_{ijk}$;
\item $\delta  q =0$.
\end{enumerate}

The first  equation follows from the analogous proposition which holds true in the local case, while the fourth   follows  from the definition of $q$, as $\delta ^2 =0$.

The other two equations require a little more effort. Let us begin with the third. 
Since $\delta ^2 \rho =0$, we have
\begin{align} 
(\delta t)_{ijk} & = -d_{\alpha_i} \phi_{ij} + d_{\alpha_i} \phi_{ik} - d_{\alpha_j} \phi_{jk} - \frac{1}{2}[\phi_{ij},\phi_{ij}] + \frac{1}{2}[\phi_{ik}, \phi_{ik}] - \frac{1}{2} [\phi_{jk}, \phi_{jk}]. \label{eqn:cociclo1}
\end{align}
On the other hand, since the l.h.s.\ of (iii) is 
\[
d_{\bar{\alpha}} q_{ijk} = d_{\alpha_i} q_{ijk} = - d_{\alpha_i} \phi_{ij} + d_{\alpha_i} \phi_{ik} - d_{\alpha_i} \phi_{jk},
\]
we just need to show that
\[
- d_{\alpha_j} \phi_{jk} - \frac{1}{2}[\phi_{ij},\phi_{ij}] + \frac{1}{2}[\phi_{ik}, \phi_{ik}] - \frac{1}{2} [\phi_{jk}, \phi_{jk}]
= - d_{\alpha_i} \phi_{jk}.
\]

By   Lemma \ref{change_d}, this is equivalent to
\[
\frac{1}{2}[\phi_{ij},\phi_{ij}] - \frac{1}{2}[\phi_{ik}, \phi_{ik}] + \frac{1}{2}[\phi_{jk}, \phi_{jk}] = - [\phi_{ij},\phi_{jk}]\,,
\]
which holds true since $\ad_{(\delta \phi)_{ijk}} = 0$ implies the identity
\[
[\phi_{ij}, \phi_{ij}] = [\phi_{ij} , \phi_{ik}] - [\phi_{ij} , \phi_{jk}].
\]

Finally, we prove that also the second equation holds true. Expanding it we obtain:
\[
d_{\alpha_i} (\rho_j - \rho_i - d_{\alpha_i}\phi_{ij} - \frac{1}{2} [\phi_{ij},\phi_{ij}] ) = d_{\alpha_j} \rho_j - d_{\alpha_i} \rho_i\ ,
\]
which is equivalent to
\[
d_{\alpha_i} \rho_j - d_{\alpha_j} \rho_j = d_{\alpha_i}^2 \phi_{ij} + \frac{1}{2} d_{\alpha_i}([\phi_{ij} , \phi_{ij}])\ .
\]

Using again Lemma \ref{change_d} and equation \eqref{curvature_1}, we have
\[
-[\phi_{ij}, \rho_j] = [\rho_i , \phi_{ij}] + \frac{1}{2}d_{\alpha_i} [\phi_{ij},\phi_{ij}].
\]
Since $d_{\alpha_i}$ is a derivation for the bracket, $\frac{1}{2}d_{\alpha_i} [\phi_{ij},\phi_{ij}] = [d_{\alpha_i}\phi_{ij},\phi_{ij}]$, and the thesis now follows since $\rho_j-\rho_i - d_{\alpha_i} \phi_{ij}$ takes values in the center of $\calL$.
\end{proof}

We can eventually introduce the following cohomology class.
\begin{defin}
 $\ob(\bar{\alpha})$ is the cohomology class  of the triple $(\{\lambda_i\},\{t_{ij}\},\{q_{ijk}\})$ in $\mathbb{H}^3 \big( X;$ $ \tau^{\geq 1}\Omega^\bullet_\cB(Z(\calL) \big)$.
\end{defin}

The definition of the triple $(\{\lambda_i\},\{t_{ij}\},\{q_{ijk}\})$ depends on the choice of a lifting triple, so we should check that two different lifting triples yield cohomologous cochains. To this end, at first we consider the case of two lifting triples having the same $\alpha_i$. Let $(\{\alpha_i\},\{\rho_{i}\},\{\phi_{ij}\})$ and $(\{\alpha_i\},\{\rho'_{i}\},\{\phi'_{ij}\})$ be  two such lifting triples, and $(\{\lambda_i\},\{t_{ij}\},\{q_{ijk}\})$ and $(\{\lambda'_i\},\{t'_{ij}\},\{q'_{ijk}\})$ the corresponding cocycles. Since the $\alpha_i$ are the same, the differences $\rho'_i-\rho_i$ and $\phi'_{ij}-\phi_{ij}$ take values in $Z(\calL)$. Then we have:
\begin{eqnarray}
\lambda'_i - \lambda_i &=& d_{\bar{\alpha}} (\rho'_i - \rho_i),\nonumber\\
t'_{ij} - t_{ij} &=& \delta(\rho'-\rho)_{ij} - d_{\bar{\alpha}}(\phi'_{ij} - \phi_{ij}),\nonumber\\
q'_{ijk} - q_{ijk} &=& \delta(\phi'-\phi)_{ijk},\nonumber
\end{eqnarray}
and the cocycles are cohomologous.

The proof for two general lifting triples follows from the following:
\begin{lemma}
Let $(\{\alpha_i\},\{\rho_{i}\},\{\phi_{ij}\})$ be a lifting triple of $\bar{\alpha}$, $(\{\lambda_i\},\{t_{ij}\},\{q_{ijk}\})$ the associated cocycle and $\{\eta_i\} \in C^0(\mathfrak U ; \Omega_\cB^1(\calL))$.
Then $(\{\alpha'_i\},\{\rho'_{i}\},\{\phi'_{ij}\})$ is another lifting triple of $\bar{\alpha}$, where:
\begin{eqnarray}
&&\alpha'_i = \alpha_i + \ad_{\eta_i},\nonumber\\
&&\rho'_i = \rho_i + d_{\alpha_i} \eta_i +\frac{1}{2}[\eta_i,\eta_i],\nonumber\\
&&\phi'_{ij} = \phi_{ij} + \eta_j - \eta_i\nonumber.
\end{eqnarray}
Its associated cocycle $(\{\lambda'_i\},\{t'_{ij}\},\{q'_{ijk}\})$ is equal to $(\{\lambda_i\},\{t_{ij}\},\{q_{ijk}\})$. 
\end{lemma}

We can now prove the first part of Theorem \ref{extthm}, i.e.:
\begin{prop}
The class $\ob(\bar{\alpha})$ vanishes if and only if there exist a holomorphic Lie algebroid extension of $\cB$ by $\calL$.
\end{prop}

\begin{proof}
Let an extension as in \eqref{ext} be given. Consider a collection $\{s_i\}$ of splittings $s_i:\cB_{|U_i} \to \cA_{|U_i}$. Out of it we construct $\alpha_{s_i}$ and $\rho_{s_i}$ as in Subsection \ref{LRalgebras}, and, moreover, we define $\phi_{s_{ij}} = s_j - s_i  \in \Omega^1_\cB(\calL)(U_{ij})$. The triple $(\{\alpha_{s_i}\},\{\rho_{s_i}\},\{\phi_{s_{ij}}\})$ is a lifting triple of $(\calL,\bar{\alpha})$. The results of Subsection \ref{LRalgebras}, together with some straightforward computations, allow us to conclude that the corresponding cocycle $(\{\lambda_i\},\{t_{ij}\},\{q_{ijk}\})$ is zero.

Conversely, assume that $\ob(\bar{\alpha}) = 0$, and let $(\{\alpha_i\} , \{\rho_{i}\} , \{\phi_{ij}\})$ be any lifting triple, with corresponding cocycle $(\{\lambda_i\},\{t_{ij}\},\{q_{ijk}\})$. As  this is a coboundary, there exist 
$$
\{a_i\} \in \check{C}^0 \big( \mathfrak U; \Omega^2_\cB(Z(\calL)) \big)\qquad\mbox{and}\qquad \{m_{ij}\}\in \check{C}^1 \big( \mathfrak U; \Omega^1_\cB(Z(\calL)) \big)
$$ 
such that
\begin{equation} \label{eq_ob_1}
\lambda_i = d_{\bar{\alpha}} a_i , \quad t_{ij} = (\delta a)_{ij} - d_{\bar{\alpha}} m_{ij} ,\quad  q_{ijk}= (\delta m)_{ijk} .
\end{equation}
Since $a_i$ and $b_{ij}$ take values in the center $Z(\calL)$, the triple $(\{\alpha_i\}, \{\rho_i - a_i \}, \{\phi_{ij} - m_{ij} \} )$ is again a lifting triple for $(\calL,\bar\alpha)$. Using again the results of Subsection \ref{LRalgebras}, on each open set $U_i$ we can define a bracket $[\cdot,\cdot]_{\alpha_i, \rho_i - a_i}$ on $(\cB \oplus \calL)_{|U_i}$. By Lemma \ref{lemma_b}, this defines a Lie algebroid extension of $\cB_{|U_i}$ by $\calL_{|U_i}$.
On the double intersections $U_{ij}$, the endomorphism of $(\cB\oplus \calL)_{|U_{ij}}$ 
\[
\phi_{ij}-m_{ij} : (b,l) \mapsto (b, l + \phi_{ij}(b) - m_{ij}(b))
\]
is defined. The second equation in \eqref{eq_ob_1} ensures that $\phi_{ij}-m_{ij}$ is a morphism of algebras 
\[
(\cB_{|U_{ij}} \oplus \calL_{|U_{ij}},[\cdot,\cdot]_{\alpha_j,\rho_j-a_j} ) \to (\cB_{|U_{ij}} \oplus \calL_{|U_{ij}},[\cdot,\cdot]_{\alpha_i,\rho_i-a_i} ),
\]
while the third equation in \eqref{eq_ob_1} guarantees that the maps $\phi_{ij} - m_{ij}$ can be used to glue   the local extensions just defined to obtain a Lie algebroid over $X$, which is an extension of $\cB$ by $\calL$.
This ends the proof.
\end{proof}

We finally come to the proof of the second part of Theorem \ref{extthm}, i.e.:
\begin{prop}
When the obstruction $\ob(\bar{\alpha})$ vanishes, the set of isomorphism classes of extensions of $\cB$ by $\calL$ inducing the coupling $(\calL,\bar{\alpha})$ is a  torsor over $\mathbb{H}^2 \big( X; \tau^{\geq 1}\Omega^\bullet_\cB(Z(\calL)) \big)$.
\end{prop}

\begin{proof}
Let an extension $\cA$ be given, choose a collection of local splittings $s_i:\cB_{U_i} \to \cA_{|U_i}$, and construct the corresponding lifting triple $(\{\alpha_{s_i}\},\{\rho_{s_i}\},\{\phi_{s_{ij}}\})$ as in the proof of the previous Proposition. Let $(\{\gamma_i\},\{\psi_{ij}\})$ be a representative of a class in $\mathbb{H}^2 \big( X; \tau^{\geq 1}\Omega^\bullet_\cB(Z(\calL)) \big)$, with $\{\gamma_i\} \in K^{2}_0$ and $\{\psi_{ij}\} \in K^{1}_1$. Then, since both $\gamma_i$ and $\psi_{ij}$ take values in $Z(\calL)$, also $(\alpha_{s_i} , \rho_{s_i} + \gamma_i , \phi_{s_{ij}} + \psi_{ij})$ is a lifting triple of $\bar{\alpha}$, and as $(\{\gamma_i\},\{\psi_{ij}\})$ is closed in $T^2$, the brackets $[\cdot,\cdot]_{\alpha_{s_i, \rho_{s_i} + \gamma_i}}$ define Lie algebroid extensions of $\cB_{|U_i}$ by $\calL_{|U_i}$ that glue with the $\phi_{s_{ij}} + \psi_{ij}$ to yield a global Lie algebroid extension of $\cB$ by $\calL$.

Conversely, given two extensions $\cA$ and $\cA'$, construct   lifting triples $(\{\alpha_{s_i}\},\{\rho_{s_i}\},\{\phi_{s_{ij}}\})$ and $(\{\alpha'_{s'_i}\},\{\rho'_{s'_i}\},\{\phi'_{s'_{ij}}\})$ out of a collection of local splittings as before. Choosing $\eta_i \in \Omega^1_\cB(Z(\calL))$ such that $\alpha'_{s'_i} - \alpha_{s_i} = \ad_{\eta_i}$, we obtain  that
\begin{eqnarray}
\gamma_i  &=& \rho'_{s'_i} - \rho_{s_i} -  d_{\alpha_{s_i}} \eta_i - \frac{1}{2}[\eta_i,\eta_i],\nonumber\\
\psi_{ij} &=& \phi'_{s'_{ij}} - \phi_{s_{ij}}\nonumber
\end{eqnarray}
take values in $Z(\calL)$ and   satisfy the necessary conditions to define a class in the hypercohomology group $\mathbb{H}^2 \big( X; \tau^{\geq 1}\Omega^\bullet_\cB(Z(\calL)) \big)$.
\end{proof}

\bigskip\section{A spectral sequence}\label{spectral}

Let $X$ be a  complex manifold. Given an extension of Lie algebroids on $X$ as in eq.~\eqref{ext}, 
one can   construct a spectral sequence which converges to the hypercohomology $\mathbb H^\bullet(X;\Omega_{ \cA}^\bullet)$.  
This is a generalization of thespectral sequence associated to an inclusion of Lie algebras $\mathfrak h \subseteq \mathfrak g$ as given in  \cite{Hoch-Serre53}. This result was generalized to the case of Lie-Rinehart algebras in \cite{Roub80}, while the case of smooth Lie algebroids was treated in \cite{mackenzie}, Chapter 7.4.
We explicitly compute the first  and   second term of this spectral sequence.

We consider an exact sequence of holomorphic Lie algebroids as in equation \eqref{ext}. We shall assume that both $\cB$ and $\calL$ are locally free $\cO_X$-modules.
\begin{defin} \label{filtration}
For $p=0,\dots,q$, we  define $F^p \Omega^q_\cA$ as the subsheaf of $\Omega_\cA^q$ whose sections are annihilated by the inner product with
$q-p+1$ sections of $\mathcal L $. 
\end{defin}

Note that $F^0 \Omega_\cA^q = \Omega_\cA^q$, $F^q\Omega_\cA^q=\Omega^q_{\cB} $, $F^{q+1} \Omega_\cA^q=0$.

In this way a decreasing filtration of the complex $\Omega_\cA^\bullet$ is defined, which, according to  \cite[Ch.~0, 13.6.4]{EGA3-I},   induces a filtration on the complex that computes the hypercohomology of the complex $\Omega_\cA^\bullet$. For the sake of clarity we give here the details of this construction.
      
\begin{lemma} $\operatorname{gr}_p \Omega^{p+q}_\cA := F^p \Omega^{p+q}_\cA/F^{p+1} \Omega^{p+q}_\cA \simeq \Omega^p_{\cB}\otimes \Omega ^q_{\mathcal L}$. 
 \label{iso}
\end{lemma}
\begin{proof} We define a morphism:
\[
j\colon  F^p \Omega^{p+q}_\cA \to  \Omega^p_{\cB}\otimes \Omega^q_{\mathcal L}
\]
by letting
\[
j(\omega) (b_1,\dots,b_p;l_1,\dots,l_q) = \omega(\bar{b}_1,\dots,\bar{b}_p,l_1,\dots,l_q),
\]
where $\omega$ is a section of $F^p \Omega^{p+q}_\cA$, $b_1,\ldots,b_p$ are sections of $\cB$, $l_1,\ldots l_q$ are sections of $\calL$, and   $\bar{b}_i$ is any section of $\cA$ whose image in $\cB$ is $b_i$. Since $\omega \in F^p \Omega^{p+q}_\cA$, this definition does not depend on the choice of $\bar{b}_i$.
The kernel of $j$ coincides with $F^{p+1} \Omega^{p+q}_\cA$,  and  this map is surjective. In fact, since both $\cB$ and $\calL$ are locally free, for all $k,p$, locally we can always find an isomorphism $F^p \Omega^{k}_\cA\simeq \bigoplus_{p'\geq p} \Omega_\cB^{p'} \otimes \Omega_\calL^{k-p'}$. 
\end{proof}
       
\begin{thm}   \label{mainDR} The filtration $F^\bullet\Omega_\cA^\bullet$ induces a spectral sequence which converges to the hypercohomology
$\mathbb H(X;\Omega_\cA^\bullet)$.The first term of this spectral sequence is 
\begin{equation} E_1^{p,q}   \simeq  \mathbb H^{q}(X, \Omega^p_{\cB}\otimes \Omega^\bullet_{\mathcal L})\,.
\label{E1} \end{equation}
\end{thm}
The differential of the complex in the r.h.s.\ of this equation is the differential $d_\calL$ induced by the trivial action of $\calL$ on $\Omega_\cB^p$. 

We make some preparations for the proof of Theorem \ref{mainDR}. By standard homological constructions (see e.g.\ \cite{Tennison}),  one can introduce injective resolutions $\mathcal C^{q,\bullet }$ of $\Omega^q_\cA$, and $F^p\mathcal C^{q,\bullet }$ of $F^p\Omega_\cA^q$, such that $F^p\mathcal C^{q,\bullet }$ is a filtration of $\mathcal C^{q,\bullet }$, and $\gr_p \mathcal C^{q,\bullet}:= F^p\mathcal C^{q,\bullet }/ F^{p+1}\mathcal C^{q,\bullet }$ is an injective resolution of $\gr_p \Omega^q_\cA$.
We consider the total complex 
\[
T^k = \bigoplus_{p+q=k} \Gamma(X,\mathcal C^{p,q})
\]
whose cohomology is the hypercohomology of $\Omega_\cA^\bullet$. Its descending filtration is defined by
\[       
F^\ell T^k = \bigoplus_{p+q=k} \Gamma(X, F^\ell \mathcal C^{p,q}).
\]
       
\begin{lemma} \label{quotient} 
One has an isomorphism
 \[
 \Gamma(X,F^\ell \mathcal C^{p,q}) /
 \Gamma(X,F^{\ell +1}\mathcal C^{p,q})\simeq \Gamma(X,\operatorname{gr}_\ell \mathcal C^{p,q}).
 \]
 \end{lemma}
 \begin{proof} As $F^p \mathcal C^{q,\bullet }/ F^{p+1}\mathcal C^{q,\bullet } \simeq 
 \gr_\ell \mathcal C^{p,q}$ one has the exact sequences
\[ 
 0 \to   \Gamma(X, F^{\ell +1}\mathcal C^{p,q} ) \to \Gamma(X, F^\ell \mathcal C^{p,q}) \to \Gamma(X, \gr_\ell\mathcal C^{p,q}) \to 
H^1(X, F^{\ell +1} \mathcal C^{p,q} )=0.
\]
\end{proof}
        
 \noindent{\em Proof of Theorem \ref{mainDR}.}
 As a consequence of Lemma \ref{quotient}, the zeroth  term of the spectral sequence given by the  filtration $F^\ell T^k$  is
 \[
 E_0^{\ell,k}  = \bigoplus_{p+q=k+\ell} \Gamma(X, \gr_\ell\mathcal C^{p,q}) .
 \]
 Recalling that the differential $d_0 \colon E_0^{\ell,k} \to E_0^{\ell,k+1}$ is induced by the differential of the complex $\Omega_\cA^\bullet$, we obtain 
 \[
 E_1^{\ell,k} \simeq \mathbb H^k(X,F^\ell \Omega_\cA^{\bullet} / F^{\ell + 1} \Omega_\cA^{\bullet}).
 \]
By plugging in the isomorphism proved in Lemma \ref{iso}, we get   Equation \eqref{E1}. \qed

 \begin{remark}  \label{local_ss} (i) One can also consider the spectral sequence associated with the filtration of the complex of sheaves $F^\bullet \Omega^\bullet_\cA$. In this case the first terms of the spectral sequence are the sheaves 
\[
 \mathscr E^{p,q}_0 = \Omega^p_\cB \otimes \Omega_\calL^q,  \quad\mathscr E_1^{p,q} = \Omega^p_\cB \otimes \mathscr G^q,
\]
where  $\cG^q$  are the cohomology sheaves  of the  complex $\Omega^\bullet_{\mathcal L}$.

\noindent
Note that the $\mathscr E^{p,q}_1$ are the sheaves associated with the presheaves
\[
U \rightsquigarrow \mathbb H^q(U,\Omega^p_{\cB}\otimes\Omega^\bullet_{\mathcal L}),
\]
but $E_1^{p,q}$ is not the vector space of global sections of $\mathscr E_1^{p,q}$.
 
(ii) Note that given the extension \eqref{ext}, we have the coupling $\bar{\alpha}: \cB \to \Out_\calD (\calL)$. This induces a representation of $\cB$ on the sheaves $\cG^q$, so that we can consider the $\cB$-forms with values in $\cG^q$. These sheaves form a complex $(\Omega_\cB^\bullet(\cG^q),d_\cB)$ and by the previous item, $\mathscr E_1^{p,q} = \Omega^p_\cB(\cG^q)$. One checks that the differential $d_1$ of the spectral sequence coincides with $d_\cB$ (cf. \cite{Hoch-Serre53} and \cite{mackenzie} for details). It follows that the   second term of the spectral sequence is formed by the cohomology sheaves 
\[
\mathscr E_2^{p,q} = \mathscr H^p(\Omega_\cB^\bullet (\cG^q), d_\cB).
\]
However, a description of the  differential $d_1$  of term $E_1^{p,q}$ will need a finer study, as it cannot be recovered directly from   $d_1$ of $\mathscr E_1^{p,q}$. The next section will be devoted to this goal.
\end{remark}
 
 \begin{example} 
 Given a Lie algebroid $\cA$, denote by $\cI$ the image of the anchor $a$ and by $\mathcal N$ the kernel of $a$.  One gets an extension 
 \[
 0 \to\mathcal N \to \cA \to \mathcal I \to 0
 \]
 of the type \eqref{ext}. Then a spectral sequence is intrinsically associated to the Lie algebroid $\cA$. Note that in general $\mathcal I$ and $\mathcal N$ are not locally free $\cO_X$-modules, but the definition of the filtration on $\cA$ makes sense also in this case, so that we have indeed a spectral sequence associated to any Lie algebroid.
 
Consider now the sheaves $\mathcal G^q$, the cohomology sheaves of the complex o sheaves $\Omega_{\mathcal N}^\bullet$ (cf. Remark \ref{local_ss}). The differential $d_1$ of the local spectral sequence $\mathscr E_1^{p,q}$ is a differential
\[
d_1: \Omega_{\mathcal I}^p \otimes \mathcal G^q \to \Omega_{\mathcal I}^{p+1} \otimes \mathcal G^q\ ,
\]
which, motivated by the next example, can be interpreted as a Gauss-Manin connection on the sheaves $\mathcal G^q$.
\end{example}

\begin{example}
Let $X,Y$ be complex manifolds and $p:X\to Y$ a submersion. Then we have the exact sequence
\[
0 \to T_p \to T_X \to p^*T_Y \to 0
\]
of vector bundles over $X$. Here $T_p$ is the sub-Lie algebroid of $T_X$ whose sections are the vector fields tangent to the fibres of $p$.
One introduces the sheaves $\mathcal H^q_{DR}(X/Y) = R^q p_* \Omega^\bullet_{X/Y}$. The Gauss-Manin connection associated to the morphism $p$ is a flat connection
\[
\nabla_{GM}: \mathcal H^q_{DR}(X/Y) \otimes \Omega_Y^p \to \mathcal H^q_{DR}(X/Y) \otimes \Omega_Y^{p+1}\ .
\]
In \cite{katz-oda} it is shown that this connection is the $d_1$ differential of the spectral sequence defined as follows. Consider the filtration defined in \ref{filtration} with $\cA= T_X$ and $\calL = T_p$. Note that $T_p$ is not an ideal in $T_X$, however,   Definition \ref{filtration} makes sense for any $\calL$ sub-Lie algebroid of $\cA$. Applying the derived functor $R p_*$ to the filtered complex of sheaves $F^\bullet \Omega_X^\bullet$, by the construction of \cite[Ch.~0, 13.6.4]{EGA3-I}, we obtain a spectral sequence converging to $R^\bullet p_*(\Omega^\bullet_X)$. The $E_1^{p,q}$ term of this spectral sequence is isomorphic to $\mathcal H^q_{DR}(X/Y) \otimes \Omega_Y^p$, and the $d_1$ differential coincides with the Gauss-Manin connection.
\end{example}


The rest of this Section is devoted  to the explicit computation 
 of the differential $d_T$ when a collection of local splittings $\mathfrak s$ is given. This will allow us to compute the $d_1$ differential of the spectral sequence, and to understand how the others differentials behave.

\begin{remark}
Let $\Xi \in \Omega_\cB^k \otimes \bigwedge^j \calL$ and $\eta \in \Omega^p_\cB \otimes \Omega^q_\calL$. The duality pairing between $\calL$ and $\Omega^1_\calL$ induces a cup-product $\Xi \smile \eta \in \Omega_\cB^{p+k} \otimes \Omega_\calL^{q-j}$, defined explicitly by the formula
\begin{eqnarray*}
&&(\Xi \smile \eta) (b_1,\ldots , b_{p+k}; l_1 , \ldots , l_{q-j})  =\\
&=& \sum_{\sigma \in \Sigma_{p, k}} (-1)^\sigma \eta \big( b_{\sigma(1)} , \ldots , b_{\sigma(p)} ; \Xi(b_{\sigma(p+1)}, \ldots ,b_{\sigma(p+k)}) , l_1 , \ldots ,l_{q-j} \big).
\end{eqnarray*}
\end{remark}

Let an extension as \eqref{ext} be given and let $s:\cB_{|U} \to \cA_{|U}$ be a local section defined over some open subset $U\subseteq X$. The splitting $s$ induces an isomorphism $\cA_{|U} \iso \cB_{|U} \oplus \calL_{|U}$, which, in turn, defines an isomorphism
\begin{equation} \label{split_forms}
\Omega^k_{\cA|U} \stackrel{s}{\to} \bigoplus_{m=0}^k \Omega_{\cB|U}^{m} \otimes \Omega_{\calL|U}^{k-m}.
\end{equation} 
Given $\xi \in \Omega^k_{\cA|U}$, we shall denote by $s(\xi)$ the corresponding element on the right hand side of (\ref{split_forms}). In particular, we shall write $s(\xi) = (s(\xi)^{0,k},\ldots, s(\xi)^{k,0})$, with $s(\xi)^{m,k-m} \in \Omega_{\cB|U}^m \otimes \Omega_{\calL|U}^{k-m}$. 
\begin{remark}
Observe that $s(\xi)^{m,k-m}$ is defined by the  formula
\begin{equation} \label{eq_spl2}
s(\xi)^{m,k-m}(b_1,\ldots,b_m; l_1,\ldots,l_{k-m}) = \xi(s(b_1), \ldots, s(b_{k-m}), l_1,\ldots, l_m).
\end{equation}
\end{remark}
Moreover, the splitting $s$ induces a $\cB_{|U}$-connection $\alpha_s$ on $\calL_{|U}$ (cf. Section 3). By a standard argument, $\alpha_s$ induces a $\cB_{|U}$-connection on $\calL_{|U}^*$ and on all   exterior powers $\Omega_{\calL|U}^p$. Explicitly,
this action is given as 
\begin{equation} \label{eq_spl4}
(\alpha_s(b) \cdot \eta)(l_1,\ldots l_p) = \alpha_s(b) \big( \eta(l_1,\dots , l_p) \big) + \sum_a (-1)^a \eta(\alpha_s(b)(l_a) , l_1,\ldots , \hat{l}_a , \ldots , l_p)
\end{equation}
for $\eta\in \Omega^p_{\calL|U}$ and $b\in \cB_{|U}$.
From this we obtain the differential 
\[
d_{\alpha_s}: \Omega_{\cB|U}^m \otimes \Omega_{\calL|U}^{k-m} \to \Omega_{\cB|U}^{m+1} \otimes \Omega_{\cB|U}^{k-m},
\]
which satisfies $d_{\alpha_s}^2 = F_{\alpha_s} \smile \bullet$, where $F_{\alpha_s}$ is the curvature of $\alpha_s$ acting on $\Omega_\calL^{k-m}$.

On the other hand, since $\calL$ is totally intransitive, we can consider the trivial $\calL$-action on $\Omega_\cB^m$. In this way we obtain a differential
\[
d_\calL : \Omega_{\cB|U}^m \otimes \Omega_{\calL|U}^{k-m} \to \Omega_{\cB|U}^{m } \otimes \Omega_{\calL|U }^{k-m+1},
\]
which satisfies $d_\calL^2 = 0$.

Now we have:
\begin{lemma} \label{lemma_conti1} For $\xi \in \Omega_{\cA|U}^k$ one has
\[
s(d_\cA \xi)^{m+1,k-m} = d_\calL \big( s(\xi)^{m,k-m} \big) + (-1)^{m+1} d_{\alpha_s} \big( s(\xi)^{m+1,k-m-1} \big) - \rho_s \smile s(\xi)^{m+2, k-m-2} .
\]
\end{lemma}

We check now how this decomposition depends on the splitting. Let $s'$ be another splitting defined over an open set $U'\subseteq X$, so that $s'-s = \phi : \cB_{U\cap U'} \to \calL_{U\cap U'}$.

\begin{lemma} \label{lemma_conti2}
Given $\xi \in \Omega_\cA^k$, and    two local splittings $s,s'$ as above,   we have
\[
s'(\xi)^{m,k-m} = s(\xi)^{m,m-k} +  \sum_a \wedge^a \phi \smile  s(\xi)^{m+a , k-m-a}  , 
\]
where $\phi=s'-s$, and $\wedge^a \phi \in \Omega_\cB^a \otimes \bigwedge^a \calL$ is defined by $(\wedge^a \phi) (b_1,\ldots b_a) = \phi(b_1) \wedge \ldots \wedge \phi(b_a)$.
\end{lemma}

Let now $\mathfrak U = \{U_i\}_{i\in I}$ be an open cover of $X$, and let $\mathfrak s= \{s_i\}$ be a collection of splittings of \eqref{ext}, with each $s_i$ defined over $U_i$. We can construct a lifting triple $(\{\alpha_{s_i}\},\{\rho_{s_i}\},\{\phi_{s_{ij}}\})$ as described in Subsection \ref{ss:holoLie}.
The  groups of   \v{C}ech cochains $K_q^p = \check{C}^q (\mathfrak U , \Omega^p_\cA)$ form a double complex with differentials $\delta_1=d_\cA$ and $\delta_2 = \delta$ the \v{C}ech differential. If  $\mathfrak U$ is good in the sense previously introduced, the cohomology of the total complex $T^k= \bigoplus_{p+q=k} K_q^p$ is isomorphic to the holomorphic Lie algebroid cohomology of $\cA$.

Now, the choice of the local splittings $\mathfrak s$ induces an isomorphism
\begin{equation} \label{eq_spl1}
\check{C}^q (\mathfrak U , \Omega_\cA^p) \stackrel{\mathfrak s}{\to} \bigoplus_{p'+r = p} \check{C}^q\big( \mathfrak U , \Omega_\cB^{p'} \otimes \Omega_\calL^r \big) .
\end{equation}
This isomorphism is explicitly given by  $\mathfrak s \{c_{i_0\cdots i_{a}}\} = \{s_{i_0}(c_{i_0\cdots i_a})\}$. We use the notation $K_q^{p,r} := \check C^q( \mathfrak U, \Omega^p_\cB \otimes \Omega_\calL^r ) $.
Under the isomorphism \eqref{eq_spl1}, the filtration takes the form 
\begin{equation} \label{eq_spl3}
F^p T^{p+q} = \bigoplus_{a=0}^{k} \bigoplus_{m= 0}^{q-a} K^{p+m,q-a-m}_{a}
\end{equation}

Now, for any element $h\in T^k$, write $h=(h_0,\ldots,h_k)$ with
$
h_a = \{(h_a)_{i_0\cdots i_s}\} \in \check{C}^s(\mathfrak U, \Omega_\cA^{k-a}).
$
Define $(\mathfrak s h)_a$ to be the image of $h_a$ under the isomorphism \eqref{eq_spl1}, namely, $\big((\mathfrak s h)_a)\big)_{i_0\cdots i_a} = s_{i_0}((h_a)_{i_0\cdots i_a})$. Moreover, we shall write $(\mathfrak s h)_a= \big( (\mathfrak s h)_a^{k-a,0},\ldots , (\mathfrak s h)_a^{0,k-a} \big)$, with the \v{C}ech cochains $(\mathfrak s h)_a^{m,k-a-m} \in K_a^{m,k-a-m} $ defined by
\begin{equation} \label{eqn_spl5}
\big((\mathfrak s h)_a^{m,k-m-a} \big)_{i_0\cdots i_a} = s_{i_0} \big( (h_a)_{i_0\cdots i_a} \big)^{m,k-m-a},
\end{equation}
where the r.h.s. of the equation is defined in equation \eqref{eq_spl2}.

We want to understand how the differential $d_T$ of the total complex $T$ behaves under the isomorphisms \eqref{eq_spl1}. To this end, we introduce the operators which will provide a decomposition of this differential.
Now, $d_T = d_\cA + \delta$, and by Lemma \ref{lemma_conti1} we see that $d_{\cA} = d_\calL + (-1)^p d_{\alpha_{\mathfrak s}} -  \rho_{\mathfrak s} \smile \bullet$, where:
\[
d_{\calL} : K^{p,q}_a \to K^{p,q+1}_a, \quad d_{\alpha_{\mathfrak s}} : K^{p,q}_a \to K^{p+1,q}_a, \quad \rho_{\mathfrak s}\smile \bullet : K^{p,q}_a \to K^{p+2,q-1}_a
\]
for $\{\eta_{i_0\cdots i_a}\} \in K^{p,q}_a$ are defined by 
\begin{eqnarray*}
d_{\calL} \{\eta_{i_0\cdots i_a}\} = \{d_\calL \eta_{i_0\cdots i_a}\}; \\
d_{\alpha_{\mathfrak s}} \{ \eta_{i_0\cdots i_a} \} = \{d_{\alpha_{s_{i_0}}} \eta_{i_0\cdots i_a}\} ;\\
\rho_{\mathfrak s} \smile \{\eta_{i_0\cdots i_a} \}= \{\rho_{s_{i_0}} \smile \eta_{i_0\cdots i_a}\} .
\end{eqnarray*}

Moreover, for every positive integer $t$, using the cochain $\{\phi_{s_{ij}}\}$ of the lifting triple associated to $\mathfrak s$, we define the operators
\[
\wedge^t \phi_{\mathfrak s} \smile \bullet : K^{p,q}_a \to K^{p+t,q-t}_{a+1}
\]
via the formula
\[
\big( \wedge^t \phi_{\mathfrak s} \smile \{ \eta_{i_0\cdots i_a} \} \big)_{j_0\cdots j_{a+1}} = \wedge^t \phi_{j_0j_1} \smile \eta_{j_1,\cdots j_{a+1}} 
\]
for $\eta \in K^{p,q}_a$, see the Lemma \ref{lemma_conti2} for the r.h.s. of the equation.

Finally, we have the \v{C}ech coboundary operator $\delta : K^{p,q}_a \to K^{p,q}_{a+1}$.

Now we can state the following:

\begin{lemma} \label{lemma_conti5}
Let $h \in T^k$. According to the decomposition \eqref{eqn_spl5}, one has 

\begin{eqnarray*}
\big( \mathfrak s (d_Th) \big)_a ^{m, k+1-a-m} &=& d_\calL \big( (\mathfrak  s h)_a^{m,k-a-m} \big) + (-1)^m  d_{\alpha_\mathfrak s} \big( (\mathfrak s h)_a^{m-1, k+1-a-m} \big) +  \\
 & - & \rho_{\mathfrak s} \smile \big( (\mathfrak s h)_a^{m-2 , k+2-a-m} \big) + (-1)^{k+a} \delta \big( (\mathfrak s h)_{a-1}^{m,k+1-a-m} \big) + \\
 & + & \sum_{t=1}^{k-s-m} (-1)^{k+a+t+1}  \wedge^t \phi_{\mathfrak s} \smile \big( (\mathfrak s h)_{a-1}^{m-t, k+1-a-m+t} \big) .
\end{eqnarray*}

\end{lemma}

\begin{proof}
After Lemma \ref{lemma_conti1}, it suffices to understand how the \v{C}ech differential $\delta$ behaves under the decomposition \eqref{eqn_spl5}. We have:
\begin{eqnarray*}
\big( (\mathfrak s \delta h)_a^{m,n} \big)_{i_0\cdots i_a} & = & s_{i_0} \big( (\delta h_{a-1})_{i_0\cdots i_a} \big)^{m,n}  \\
 & = & - s_{i_0} \big( (h_{a-1})_{i_1,\cdots i_a} \big) + \sum_{\nu=1}^a (-1)^\nu s_{i_0} \big( (h_{a-1})_{i_0\cdots \hat{i}_\nu \cdots i_a}\big)^{m,n} \ . 
\end{eqnarray*}
On the other hand we have:
\begin{eqnarray*}
\delta \big( (\mathfrak s h)_{a-1}^{m,n} \big)_{i_0\cdots i_a} & = & - s_{i_1} \big( (h_{a-1})_{i_1,\cdots i_a} \big) + \sum_{\nu=1}^a (-1)^\nu s_{i_0} \big( (h_{a-1})_{i_0\cdots \hat{i}_\nu \cdots i_a}\big)^{m,n}\ .
\end{eqnarray*}
Using Lemma \ref{lemma_conti2} we conclude.
\end{proof}

Let us rewrite this as follows. Denote by $\Delta^m_{\mathfrak s}$ the composition of the isomorphism $T^k \stackrel{\mathfrak s}{\iso} \bigoplus_{a,b}K_a^{b,k-a-b}$ with the projection onto the subspace $\bigoplus_a K_a^{m,k-a-m} = \gr^m T^{k}$. Then the equation of Lemma \ref{lemma_conti5} becomes
\begin{align} 
\Delta_{\mathfrak s}^m (d_T h)\  =\  &  d_{\calL} (\Delta^m_{\mathfrak s} (h)) + (-1)^m d_{\alpha_{\mathfrak s}}(\Delta_{\mathfrak s}^{m-1}(h)) - \rho_{\mathfrak s} \smile \Delta^{m-2}_{\mathfrak s}(h) \nonumber \\
& + (-1)^k \delta(\Delta^m_{\mathfrak s}(h)) + \sum_t (-1)^{k+t+1} \wedge^t \phi_{\mathfrak s} \smile \Delta^{m-t}_{\mathfrak s}(h)\ . \label{eqn_spl100}
\end{align}

Moreover, for each  positive integer $a$, we can introduce the operators $D_{\mathfrak s}^a : \gr^m T^k \to \gr^{m+a} T^{k+1}$ such that for $h=\Delta_{\mathfrak s}^p h$ one has $\Delta_{\mathfrak s}^{p+a}\big( d_T(\Delta^p_{\mathfrak s}(h)) \big) = D^a_{\mathfrak s}(\Delta^p_{\mathfrak s} h)$. Explicitly, we have
\[
D^0_{\mathfrak s} = d_\calL + (-1)^k \delta\ , \qquad D^1_{\mathfrak s} = (-1)^m d_{\alpha_{\mathfrak s}} + (-1)^{k} \phi_{\mathfrak s} \smile \bullet\ ,  
\] 
\[
D^2_{\mathfrak s} = - \rho_{\mathfrak s} \smile \bullet + (-1)^{k+1} \wedge^2 \phi_{\mathfrak s} \smile \bullet\ , \qquad D^a_{\mathfrak s} = (-1)^{k+a} \wedge^a \phi_{\mathfrak s} \smile \bullet\ .
\]
Using these operators, we can write the total differential in the  compact form
\begin{equation} \label{eqn_spl111}
\Delta_{\mathfrak s}^m(d_T h) = \sum_a D^a_{\mathfrak s}(\Delta_{\mathfrak s}^{m-a}(h) ) 
\end{equation}

Note that, according to \eqref{eq_spl3}, $h\in F^pT^{p+q}$ if and only if
\begin{equation*} 
(\mathfrak s h)_a^{m , p+q-a-m} = 0 \qquad \text{for }\ \  m < p.
\end{equation*}
Then using the $\Delta^m_{\mathfrak s}$ operators, this becomes:
\begin{equation} \label{eqn_spl6}
h \in F^pT^{p+q} \ \ \text{if and only if} \ \ \Delta_{\mathfrak s}^m(h) = 0 \ \ \text{for }\ m<p .
\end{equation}

Now, recall that the terms of the spectral sequence $E_r^{p,q}$ may be defined (cf. \cite{Voisin}) as the quotient $E_r^{p,q} = Z_r^{p,q} / B_r^{p,q}$, where
\[
Z_r^{p,q} = \{\  x \in F^p T^{p+q} \ \ |\ \  d_T x \in F^{p+r}T^{p+q+1}\  \}
\]
and 
\[
B_r^{p,q} \ =\  Z_{r-1}^{p+1,q-1}\  +\  d_T Z_{r-1}^{p-r+1 , q+r-2} \ \subseteq Z_r^{p,q},
\]
while the differentials $d_r:E_r^{p,q} \to E_r^{p+r,q+1-r}$ are induced by $d_T$. 
By this definition and  \eqref{eqn_spl6}, we see that $h \in Z_r^{p,q}$ is a representative of a class $[h] \in E_r^{p,q}$ if and only if $\Delta^m_{\mathfrak s}(h) = 0$ for $m<p$ and $ \Delta_{\mathfrak s}^m (d_Th) = 0$ for $m<p+r$.

For $r=0$  we obtain
\[
Z_0^{p,q} = F^p T^{p+q}, \quad B_0^{p,q} = F^{p+1}T^{p+q},
\]
so that $E_0^{p,q} = \gr^p T^{p+q} \stackrel{\mathfrak s}{\iso} \bigoplus K_a^{p,q-a}$. To compute the differential $d_0$, let $h\in F^pT^{p+q}$; then equation \eqref{eqn_spl6} is satisfied, and from Lemma \ref{lemma_conti5} we obtain that $d_0[h]$ is the class in $E_0^{p,q}$ of
\begin{equation} \label{eqn_spl11}
\Delta^p_{\mathfrak s}(d_Th) = D^0_{\mathfrak s}(\Delta^p_{\mathfrak s}(h)).
\end{equation}
Now note that $D^0_{\mathfrak s}= d_\calL + (-1)^{p+q} \delta$ does not depend on the choice of $\mathfrak s$ and coincides with the total complex associated to the double complex $K_\bullet^{p,\bullet} = \check{C}^\bullet(\mathfrak U, \Omega^p_\cB \otimes \Omega_\calL^\bullet)$. So the term $E_1^{p,q}$ of the spectral sequence is isomorphic to $\mathbb H^p(X; \Omega^p_\cB \otimes \Omega^\bullet_\calL)$; this yields a hands-on proof of Theorem \ref{mainDR}. 

We can now compute the differential $d_1:E_1^{p,q} \to E_1^{p+1,q}$. Let $h \in Z_1^{p,q}$. Since $B_1^{p+1,q} = Z_0^{p+2,q-1} + d_T Z_0^{p+1,q-1} = F^{p+2}T^{p+q+1} + d_T F^{p+1}T^{p+q}$, we have
\[
d_Th = \Delta_{\mathfrak s}^{p+1} (d_T h) \quad \text{mod } \ B_1^{p+1,q}\ .
\]
So from equation \eqref{eqn_spl111} we obtain:
\[
d_1[h] = [D^1_{\mathfrak s}(\Delta_{\mathfrak s}^p(h)) + D^0(\Delta_{\mathfrak s}^{p+1}(h)]
\]
and, since $D^0(\Delta_{\mathfrak s}^{p+1}(h))$ is a coboundary for the $d_0$ differential, we have
\[
d_1[h] = [D^1_{\mathfrak s}(\Delta_{\mathfrak s}^p(h))].
\]

So, we can identify the $E_2$ term of the spectral sequence with the cohomology of the complex $\mathbb{H}^{q}(X; \Omega^p_\cB \otimes \Omega^\bullet_\calL) \stackrel{D^1_{\mathfrak s}}{\to} \mathbb{H}^{q}(X; \Omega^{p+1}_\cB \otimes \Omega^\bullet_\calL)$. One should check that the differential $D^1_{\mathfrak s}$ is well defined, i.e., it does not depend on the choice of the collection of splittings $\mathfrak s$. In fact, we have:

\begin{lemma}
Let $\mathfrak s$ and $\mathfrak s'$ be two collections of splittings over the same open cover $\mathfrak U$. Let $\psi_i:\cB_{|U_i} \to \calL_{|U_i}$ be the differences $\psi_i = s'_i - s_i$.
Then, for any $\xi \in \gr^p T^{p+q}$ with $D^0(\xi) = 0$, we have
\[
D^1_{\mathfrak s'} \xi - D^1_{\mathfrak s} \xi = D^0(\eta),
\]
for $\eta \in \gr^{p+1} T^{p+q}$ defined as $\eta = \psi \smile \xi$.  
\end{lemma}

The further terms of the spectral sequence become quite complicated, so that we just give a hint of how one can describe the $d_2$ differential. Let $h \in F^pT^{p+q}$ be a cochain defining an element $[h] \in E_2^{p,q}$. This means that $h \in Z^{p,q}_2$, and by \eqref{eqn_spl6} we have 
\[
D^0_{\mathfrak s} \big( \Delta^p_{\mathfrak s} (h) \big) = 0, \qquad D^1_{\mathfrak s}(\Delta^p_{\mathfrak s}(h)) = (-1)^p D^0_{\mathfrak s} (\Delta^{p+1}_{\mathfrak s}(h)).
\]
Then 
\[
d_Th = D^2_{\mathfrak s}(\Delta^p_{\mathfrak s}(h)) + D^1_{\mathfrak s}(\Delta_{\mathfrak s}^{p+1}(h)) + D^3_{\mathfrak s}(\Delta_{\mathfrak s}^p(h)) + D^2_{\mathfrak s}(\Delta_{\mathfrak s}^{p+1}(h) \quad \text{mod}\ \ B_2^{p+2,q-1}
\]
and the terms on the r.h.s.\ of the equation yield a representative of $d_2[h]$.

\bigskip\section{Example: cohomology of the Atiyah algebroid of a line bundle}\label{linebundles}

We apply the results of the previous sections to study the cohomology of the Atiyah algebroid  $\calD_{\mathscr M}$ of a line bundle  $\mathscr M$.   In this case the sequence \eqref{atiyah} takes the form
\begin{equation} \label{extlb}
0\to \cO_X \to \mathcal{D}_{\mathscr M} \to \Theta_X \to 0  \ , 
\end{equation}
where $\cO_X$ is regarded as a bundle of abelian Lie algebras.
Since $\cO_X$ has rank one, the filtration has only one nontrivial term, i.e.:
\[
F^k \Omega^k_{\calD_{\mathscr M}} \simeq \Omega^k_{X},
\]
and the only nontrivial graded objects are
\begin{equation} \label{eq:atd}
\gr_{k-1} \Omega^k_{\calD_{\mathscr M}} \simeq \Omega^{k-1}_X, \qquad \gr_k \Omega_{\calD_{\mathscr M}}^k \simeq \Omega^k_X .
\end{equation}

One can compute the cohomology of $\calD_{\mathscr M}$ directly. To this end observe that, denoting by $T^k_{\calD_{\mathscr M}}$ the complex computing the hypercohomology $\mathbb H^q(X; \Omega_{\calD_{\mathscr M}}^\bullet)$, and by $T^\bullet_{\Theta_X}$ the complex computing the hypercohomology $\mathbb H^q(X; \Omega_X^\bullet) = H^q(X;\mathbb C)$, from \eqref{eq:atd} we can deduce that $F^kT^k_{\calD_{\mathscr M}} = T^k_{\Theta_X} $ and   $\gr_k T^k _{\calD_{\mathscr M}} = T^{k-1}_{\Theta_X}$. From this we obtain the  exact sequence of complexes
\begin{equation} \label{eqn:last}
0 \rightarrow  T_{\Theta_X}^\bullet \rightarrow T^\bullet_{\calD_{\mathscr M}} \rightarrow T_{\Theta_X}^{\bullet -1} \rightarrow 0.
\end{equation}

\begin{prop} \label{thm:coho_atiyah}
The short exact sequence \eqref{eqn:last} induces a long exact sequence in cohomology
\begin{equation}\label{long}
\cdots \rightarrow  H^{q}(X;\mathbb C) \rightarrow \mathbb H^q(X; \Omega^\bullet_{\calD_{\mathscr M}}) \rightarrow H^{q-1}(X;\mathbb C) \rightarrow  H^{q+1} (X; \mathbb C)\to\cdots
\end{equation}
whose connecting morphism is given by the cup product with the first Chern class of ${\mathscr M}$.
In particular, we obtain 
\[
\mathbb H^q(\calD_{\mathscr M};\C) \simeq \frac{H^q(X;\C)}{\Image \gamma_{\mathscr M,q-2}} \oplus 
\ker \gamma_{\mathscr M,q-1},
\]
where
\begin{eqnarray*}  \gamma_{\mathscr M,q} \colon H^q(X,\C) & \to & H^{q+2}(X,\C) \\
\gamma_{\mathscr M,q} (x) & = & c_1({\mathscr M}) \smile x.
\end {eqnarray*}  
\end{prop}

\begin{proof}
Let  $(\{\alpha_{s_i}\},\{\rho_{s_i}\}, \{\phi_{s_{ij}}\})$ be a  lifting triple associated with a collection
$\mathfrak s = \{s_i\}$  of splittings of the exact sequence \eqref{extlb}. Note that giving a splitting of this sequence is equivalent to giving a connection on $\mathscr{M}$. Since we can always locally define flat connections on a line bundle, we can assume that $\rho_{s_i} = 0$. Moreover, the extension is abelian, and the action of $\Theta_X$ on $\cO_X$ coincides with the standard one, so that the $\alpha_{s_i}$ are the restriction to $U_i$ of $\alpha_{std}$, the standard action of $\Theta_X$ on $\cO_X$. Finally, the cocycle $\{\phi_{s_{ij}}\}$ is $\delta$-closed, and its cohomology class in $H^1(X; \Omega^1_X)$ is the first Chern class of ${\mathscr M}$. 

Let $c \in T^{q-1}_{\Theta_X}$ be a representative of a class in $H^{q-1}(X;\mathbb C)$, and denote by $\sigma$ the connecting morphism. Then $\sigma([c]) = [d_{\mathcal D_{\mathscr M}} x]$, for $x$ any lifting of $c$ to $T^q_{\mathcal D_{\mathscr M}}$.
In particular we can choose $x=(0,c)$, where  we are using $\mathfrak s$ to identify $T^q_{\mathcal D_{\mathscr M}}$ with $T^q_{\Theta_X} \oplus T^{q-1}_{\Theta_X}$.
So we have
\[
d_{\calD_{\mathscr M}} x =  \delta c + d_{\alpha_{std}} c + \phi_{\mathfrak s} \smile c\ .
\]
As $d_{\alpha_{std}}$ coincides with $d_X$, the de Rham differential of $X$, and $c$ is closed in $T^{q-1}_{\Theta_X}$, we have   $\delta c + d_{\alpha_{std}} c = 0$, and we obtain   $d_{\calD_{\mathscr M}} x = \phi_{\mathfrak s} \smile c$. Since $\phi_{\mathfrak s}$ is a representative of the Chern class of $\mathscr M$, the proposition is proved.
\end{proof}

As in this case $\Omega_\calL^\bullet$ is the two term complex $0\to \cO_X \stackrel{0}{\rightarrow} \cO_X\to 0 $, we have the isomorphisms
\[
E_1^{p,q} \simeq \mathbb H^q(X;\Omega_\calL^\bullet \otimes \Omega_\cB^p) \simeq H^q(X;\Omega^p_X) \oplus H^{q-1}(X;\Omega_X^{p}) .
\]

Using the computations made in the previous section, we obtain:
\begin{prop} 
(i) For $(\xi_q,\xi_{q-1}) \in H^q(X;\Omega^p_X) \oplus H^{q-1}(X;\Omega^p_X)\simeq E_1^{p,q}$, one has
\[
d_1(\xi_q,\xi_{q-1}) = (d_X \xi_q + c_1(\mathscr M) \smile \xi_{q-1} , d_X \xi_{q-1}) .
\]

(ii)  If $X$ is a compact K\"ahler manifold, the spectral sequence degenerates at the second step.

Moreover, the cohomology $\mathbb H^k (X; \Omega^\bullet_{\mathcal D_{\mathscr M}} )$ inherits a Hodge structure of weight $(k-1,k)$.
\end{prop}

\begin{proof}
Part (i) is  clear. 

To prove (ii), observe that if $X$ is a compact K\"ahler manifold, one has the Hodge decomposition for $X$. This implies that the spectral sequence associated with the b\^ete filtration of $\Omega_X^\bullet$ degenerates at the first step. This entails that for any \v{C}ech cocycle $\xi_q \in \check{C}^q(\mathfrak U, \Omega^{p}_X)$, the \v{C}ech cochain $d_X \xi_q$ is actually a $\delta$-coboundary.

This in turn implies  that the differential $d_1$ of the spectral sequence satisfies 
\[
d_1[(\xi_q,\xi_{q-1})] = [(c_1(\mathscr M) \smile \xi_{q-1} , 0)]\ .
\]
We then have 
\[
E_2^{p,q} = \frac{H^q(X;\Omega^p_X)}{\Image\gamma_{\mathscr M,q-1,p-1}} \oplus \Kernel \gamma_{\mathscr M,q-1,p},
\]
where by $\gamma_{\mathscr M,q,p}$ we denote the cup product 
\[
\gamma_{\mathscr M,q,p} = c_1(\mathscr M) \smile \bullet : H^q(X; \Omega_X^p) \to H^{q+1}(X;\Omega_X^{q+1})\ .
\]
On the other hand, we know that the spectral sequence converges to 
\[
\mathbb H^{k}(X; \Omega_{\calD_{\mathscr M}}^\bullet) = \frac{H^k(X; \mathbb C)}{\Image \gamma_{\mathscr M,q-2}} \oplus \Kernel \gamma_{\mathscr M,q-1}\ .
\]

Now, the Hodge decomposition for $X$ induces the decompositions
\begin{equation}\label{deco_hodge}
\frac{\mathbb H^k(X; \Omega_X^\bullet)}{\Image \gamma_{\mathscr M,k-2}} = \bigoplus_{p+q=k} \frac{H^q(X;\Omega^p_X)}{\Image\gamma_{\mathscr M,q-1,p-1}}\ , \quad
\Kernel \gamma_{\mathscr M,k-1} = \bigoplus_{p+q = k-1} \Kernel \gamma_{\mathscr M,q,p}\ ,
\end{equation}
from which we obtain $\mathbb H^k(X; \Omega^\bullet_{\mathcal{D}_{\mathscr M}}) = \bigoplus_{p+q=k} E_2^{p,q}$. Thus the spectral sequence degenerates at the second term.

Moreover, since $c_1(\mathscr M)$ belongs to $H^2(X; \mathbb Z)$, we have the isomorphisms 
\[
\overline{\Kernel \gamma_{\mathscr M,q,p}} = \Kernel \gamma_{\mathscr M,p,q} \ , \qquad 
\overline{\Image \gamma_{\mathscr M,q,p}} = \Image \gamma_{\mathscr M,p,q}\ ,
\]
so the decompositions \eqref{deco_hodge} define Hodge structures of weight $k$ and $k-1$, respectively. 
\end{proof}

\bigskip      
 
\frenchspacing

\def\cprime{$'$} \def\cprime{$'$} \def\cprime{$'$} \def\cprime{$'$}

\end{document}